\documentclass[12]{amsart}
\usepackage[utf8]{inputenc}   
\usepackage{color}
\usepackage{dsfont}
\usepackage{nicefrac}
\usepackage{url}
\usepackage{hyperref}
\usepackage{mathtools}
\usepackage{anysize}
\usepackage{geometry}
\marginsize{2.5cm}{2.5cm}{2.5cm}{2.5cm}
\numberwithin{equation}{section}     
\newtheorem{theorem}{Theorem}[section]
\newtheorem{lemma}{Lemma}[section]
\newtheorem{corollary}{Corollary}[section]
\newtheorem{remark}{Remark}[section]
\newcommand{\<}{\langle}
\renewcommand{\>}{\rangle}
\newcommand{\ud}{\mathrm{d}}
\begin{document}
\title{Limit behavior of the invariant measure for Langevin dynamics}
\author{Gerardo~Barrera}
\address{University of Helsinki, Department of Mathematical and Statistical Sciences. PL 68, Pietari Kalmin katu 5.
Postal Code: 00560. Helsinki, Finland.}
\email{gerardo.barreravargas@helsinki.fi}
\begin{abstract}
In this manuscript, we consider the Langevin dynamics on $\mathbb{R}^d$ with an overdamped vector field and driven by multiplicative Brownian noise  of small amplitude $\sqrt{\epsilon}$, $\epsilon>0$. Under suitable assumptions on the vector field and the diffusion coefficient, it is well-known that it possesses a unique invariant probability measure $\mu^{\epsilon}$. As $\epsilon$ tends to zero, we prove that the probability measure
$\epsilon^{\nicefrac{d}{2}} \mu^{\epsilon}(\sqrt{\epsilon}\ud x)$ converges in the $p$-Wasserstein  distance for $p\in [1,2]$ to a Gaussian measure with zero-mean vector and non-degenerate covariance matrix which solves a Lyapunov matrix equation.
Moreover, the error term is estimated.
We emphasize that generically no explicit formula for $\mu^{\epsilon}$ can be found. 
\end{abstract}
\maketitle
\markboth{Invariant probability measure}{Gaussian approximation}

\section{\textbf{Introduction}}
\subsection{\textbf{The overdamped Langevin dynamics}}
Random dynamical systems  arise in the modeling of a (realistic) physical system
subject to noise perturbations from its surrounding environments or
from intrinsic uncertainties associated with the system. 
The Langevin dynamics was introduced by P. Langevin in 1908 in his celebrated article 
{\it Sur la th\'eorie du mouvement brownien,} 
C. R. Acad. Sci. Paris 146, pp.
530--533.
It is perhaps one of the most 
popular models in molecular systems. For details about the history of the Langevin equation, see~\cite{POPI}.
For a phenomenological treatment, we recommend the 
monography~\cite{COK}.

In the last decades, there have been many applications
of Markov chain Monte Carlo methods to complex  systems in
 Computer Science and Statistical Physics.
Since  
sampling 
high-dimensional distributions is typically a difficult task,
the use of stochastic equations  for sampling has become important in many applications such as artificial intelligence and  Bayesian algorithms.
Stochastic algorithms based on Langevin equations have been proposed to simulate and improve the rate of convergence  to limiting distributions. 
For further details we refer 
to~\cite{Dalalyan2017}, \cite{Duncan2017},
\cite{Eberle2019}, \cite{EberleTAMS2019},
 \cite{Hwang193}, \cite{Levi2013}, \cite{Wu2014} and the references therein.

Differential equations subject to small noise perturbations are one of the classic directions of modern mathematical physics.
Let $\epsilon\in (0,1]$ be a parameter that measures the perturbation strength
and let $(B_t)_{t \geq 0}$ be a standard Brownian motion on $\mathbb{R}^d$. 
For any (deterministic) $x\in \mathbb{R}^d$ we consider the unique strong solution $(X^{\epsilon}_t(x))_{t \geq 0}$ of the following stochastic differential equation (SDE for short) on $\mathbb{R}^d$
\begin{equation}\label{over}
\left\{
\begin{array}{r@{\;=\;}l}
\ud X^{\epsilon}_t(x)&-F(X^{\epsilon}_t(x))\ud t +
\sqrt{\epsilon}\sigma(X^\epsilon_t(x))\ud B_t\quad  
\textrm{  for any }\quad t\geq 0,\\
X^{\epsilon}_0(x) & x,
\end{array}
\right.
\end{equation}
where the vector field $F\in\mathcal{C}^2(\mathbb{R}^d,\mathbb{R}^d)$ and the diffusion coefficient  $\sigma\in
\mathcal{C}^2(\mathbb{R}^d,\mathbb{R}^{d\times d})$ satisfy the following assumptions.
We assume that $0_d$ is a fixed point for $F$, i.e., $F(0_d)=0_d$ and the
following hypotheses for $F$.

\noindent
\textbf{Bakry--\'Emery condition:}
there exists a positive constant $\delta$  such that
\begin{equation}\label{H}
\tag{{\bf{A}}}
\<F(x_1)-F(x_2),x_1-x_2\>\geq \delta \|x_1-x_2\|^2\quad \textrm{ for any }\quad x_1,x_2\in \mathbb{R}^d,
\end{equation}
where
$\<\cdot,\cdot\>$ denotes the standard inner product on $\mathbb{R}^d$ and
$\|\cdot\|$ denotes the standard Euclidean  norm on $\mathbb{R}^d$. 

\noindent
\textbf{Exponential growth condition:}
there exist positive constants $c_0$ and $c_1$ satisfying
\begin{equation}
\tag{\textbf{B}}
\label{C3}
\|D^2F(x)\| \leq c_0e^{c_1\|x\|^2}\quad \textrm{ for any }\quad x \in \mathbb{R}^d,
\end{equation}
where
$D^2 F$ denotes the second order derivative of $F$.

For the diffusion coefficient 
$\sigma$ we assume the following standard hypotheses.

\noindent
\textbf{Lipschitz continuity:} there exists a positive constant $\ell$ such that 
\begin{equation}\label{eq:lipt}
\tag{{\bf{C}}}
\|\sigma(x)-\sigma(x_0)\|_{\mathrm{F}}\leq \ell \|x-x_0\| \quad\textrm{ for all }\quad x, x_0\in \mathbb{R}^d,
\end{equation}
where $\|\cdot\|_{\mathrm{F}}$ denotes the Frobenius norm.

\noindent
\textbf{Ellipticity:} 
there is a positive constant $\kappa$ such
that
\begin{equation}\label{eq:elip}
\tag{{\bf{D}}}
  \<\sigma(x_0) \sigma^*(x_0)x, x\>\geq \kappa\|x\|^2
\quad\textrm{ for all }\quad x, x_0\in \mathbb{R}^d,
\end{equation}
where $*$ denotes the transpose operator.

Hypotheses~\eqref{H} and~\eqref{eq:lipt} imply the monotone condition~(3.14) given in Theorem~3.5, p.~58 
in~\cite{Mao2008}, and hence the existence and uniqueness of the unique strong solution of~\eqref{over}.
Along this manuscript, let $(\Omega,\mathcal{F},\mathbb{P})$ be a  complete probability space 
where~$\eqref{over}$ is defined and denote by $\mathbb{E}$ the expectation with respect to $\mathbb{P}$.

\subsection{\textbf{Invariant distribution}}
Existence of invariant measures for stochastic processes are an important feature in probability
theory and mathematical physics; and typically they are not so easy to describe explicitly.
By Theorem~3.3.4, p.~91 in~\cite{KUL}, it is not hard to verify that  Hypotheses~\eqref{H}, \eqref{eq:lipt} 
and~\eqref{eq:elip} yield the existence and uniqueness of an invariant (absolutely continuous with respect to the Lebesgue measure on $\mathbb{R}^d$) probability measure $\mu^{\epsilon}$ for the stochastic dynamics~\eqref{over}. 
If in addition, 
$F(x)=\nabla V(x) + b(x)$ and $\sigma(x)=I_d$ for any $x\in \mathbb{R}^d$,
where $I_d$ is the identity matrix of dimension $d\times d$,
$V:\mathbb{R}^d\rightarrow \mathbb{R}$ is a scalar function and $b:\mathbb{R}^d\rightarrow \mathbb{R}^d$
is a vector field
which satisfies the divergence-free condition
\begin{equation}\label{eq:dvfree}
\sum_{j=1}^d \frac{\partial}{\partial x_j}\left(b(x)\exp(-(\nicefrac{2}{\epsilon})V(x))\right)=0\quad \textrm{ for any }\quad x=(x_1,\ldots,x_d)\in \mathbb{R}^d,
\end{equation}
one can verify that
$\exp(-(\nicefrac{2}{\epsilon})V(x))\ud x$
is a stationary measure for the random dynamics~\eqref{over}. However,  it might not be a probability measure. 
Under some appropriate assumptions on $V$ for $\|x\|\gg 1$,
the unique invariant probability measure $\mu^{\epsilon}$ 
of~\eqref{over} is of the Gibbs type
\begin{equation}\label{eq:gibbs}
\mu^{\epsilon}(\ud x) = 
\frac{\exp(-(\nicefrac{2}{\epsilon})V(x))}{\mathcal{Z}^{\epsilon}}\ud x,
\end{equation}
where $\mathcal{Z}^\epsilon$ is
the so-called partition function (normalizing constant).
See for instance Chapter 2, Convergence of the Langevin process, p.~21-23 in~\cite{VI}.
Using the Laplace Method (Saddle-point Method), asymptotics as 
$\epsilon\to 0^+$ for the density of $\mu^{\epsilon}$
can be carried out, see for instance~\cite{AHWANG} 
and~\cite{Hwang1980}.

If we drop the free-divergence condition~\eqref{eq:dvfree} and replace it by the 
transversal condition 
$\<\nabla V(x),b(x)\>=0$ for all $x\in \mathbb{R}^d$, 
in~\cite{Sheu1986} for additive noise
 it is shown a beautiful expansion on $\epsilon$ for the density of  $\mu^{\epsilon}$.
However, this expansion requires smoothness of the so-called  Freidlin--Wentzell quasipotential.
The latter is a nontrivial mathematical problem since it is expressed by a variational principle.
Using calculus of variations, in~\cite{DayDarden1985} it is shown various results about the smoothness of the quasipotential under the assumptions of  smoothness, boundedness and ellipticity of the coefficients 
of~\eqref{over}. 
In Section~5 of~\cite{Day1987} it is proved that the asymptotic expansion given in~\cite{Sheu1986} remains valid  in any open set in which the quasipotential is $\mathcal{C}^2$.
For additive noise, and bounded and dissipative vector field $F$, in~\cite{Mikami1988},
by way of Watanabe's theory and Malliavin calculus, an asymptotic expansion of $\mu^\epsilon$ has been proved.
Later, in~\cite{Mikami} it is shown that $\mu^\epsilon$ can be expanded in Wentzel--Kramers--Brillouin (W.K.B.) type, as $\epsilon \to 0^+$, in the set in which the quasipotential is of $\mathcal{C}^\infty$-class and each coefficient which appears in the expansion is of $\mathcal{C}^\infty$-class. More recently, in~\cite{Biswas2009}, using control theoretic methods, it is proved that $\mu^\epsilon(\ud x)\approx \exp(-\nicefrac{V_*(x)}{\epsilon})$, $\epsilon\ll 1$, where  $V_*$ is characterized as the optimal cost of a deterministic control problem. Nevertheless, the control problem is not easy to solve explicitly.

In~\eqref{over} we consider  multiplicative noise, and no transverse condition on the vector field $F$ is assumed. Moreover, we do not
need that the Gibbs measure~\eqref{eq:gibbs} remains stationary, and no smoothness on
$\mu^\epsilon$ and the Freidlin--Wentzell quasipotential are needed.
We remark that \textit{generically} it is not possible to compute an explicit formula for $\mu^{\epsilon}$. 

\subsection{\textbf{Informal result}}
Our goal is to prove that the probability  $\epsilon^{\nicefrac{d}{2}}
\mu^{\epsilon}(\sqrt{\epsilon} \ud x)$ has a Gaussian shape in 
the small noise limit. To be more precise, under 
Hypotheses~\eqref{H}, \eqref{C3}, \eqref{eq:lipt} 
and~\eqref{eq:elip},
it follows that the probability measure 
\begin{equation}\label{eq:scala}
\epsilon^{\nicefrac{d}{2}}
\mu^{\epsilon}(\sqrt{\epsilon} \ud x)
\end{equation}
converges in the $p$-Wasserstein ($p\in [1,2]$) to a Gaussian $\mathcal{N}$ distribution with zero-mean vector and covariance matrix given by the unique solution $\mathbb{X}$ of the Lyapunov matrix equation
\begin{equation}\label{eq:mlyaeq}
DF(0_d)\mathbb{X}+\mathbb{X}(DF(0_d))^*=\sigma(0_d)(\sigma(0_d))^*.
\end{equation}
Generically, it is hard to find an explicit formula for the solution of  
the~\eqref{eq:mlyaeq}. Nevertheless, it can be estimated via numerical algorithms, see for instance~\cite{Benner}, \cite{Hodel} and the references therein.
More precisely, it is shown an asymptotic expansion (in the Wasserstein distance) of $\mu^\epsilon$ as follows
\begin{equation}\label{eq:cuantif}
\frac{\mathcal{J}^{\epsilon}}{\sqrt{\epsilon}} =\mathcal{N}+\mathcal{O}(\sqrt{\epsilon})\quad \textrm{ for }\quad \epsilon \to 0^+,
\end{equation}
where $\mathcal{J}^{\epsilon}$ denotes a random variable with law $\mu^{\epsilon}$.

We anticipate that the proof of~\eqref{eq:cuantif} does not rely on explicit computations of the distribution $\mu^\epsilon$. It is based
 on the linearization of the nonlinear dynamics around the stationary
point $0_d$. It is not hard to see that the resulting linear process has the target Gaussian as invariant distribution. It is then necessary to control the
difference between this linear process and the nonlinear dynamics. This is done using
the so-called synchronous coupling techniques with the help 
of Hypotheses~\eqref{H}, \eqref{C3} and~\eqref{eq:lipt}.
The proof of~\eqref{eq:cuantif} is \textit{purely dynamic} and 
it does not require techniques as
Malliavin calculus, large deviation theory for SDEs as 
in~\cite{Freidlin}, smoothness of the quasipotential, smoothness of the density $\mu^\epsilon$,
analysis of the infinitesimal generator or the  W.K.B. expansion.

Quantitative bounds on the rate of convergence of Markov processes to their limiting distribution
are an important and widely studied topic, particularly in the context of Markov chains, see for 
instance~\cite{Durmus}, \cite{Gareth}, \cite{Madras} and the references therein.
We  \textit{quantify} in the Wasserstein distance  the implicit error term given in~\eqref{eq:cuantif}.
We point out that
the critical regime analyzed in Section~5.1 
of~\cite{Arapostathis} implies for additive noise the total variation convergence of~\eqref{eq:scala} to a Gaussian distribution. However, it seems hard to obtain bounds for the total variation error term, even under our assumptions on $F$ and $\sigma$.

\subsection{\textbf{Wasserstein distance}}
Let $\mathcal{P}$ be the set of probability measures in the measurable space $(\mathbb{R}^d,\mathcal{B}(\mathbb{R}^d))$, where $\mathcal{B}(\mathbb{R}^d)$ denotes the Borel $\sigma$-algebra of $\mathbb{R}^d$.
For $p\geq 1$ we define
\[
\mathcal{P}_p:=\left\{\mu\in \mathcal{P}: \int_{\mathbb{R}^d}\|x\|^p\mu(\ud x)<\infty \right\},
\]
the space of probability measures with finite $p$-moment. For any $\mu,\nu\in \mathcal{P}$ we say that a probability measure ${\pi}_*$ in the measurable space
$(\mathbb{R}^d \times \mathbb{R}^d,\mathcal{B}(\mathbb{R}^d\times \mathbb{R}^d))$ is a {\em coupling} between $\mu$ and $\nu$ if the marginals of ${\pi}_*$ are $\mu$ and $\nu$, that is, 
for any $B\in \mathcal{B}{(\mathbb{R}^d)}$ it follows that
$
{\pi}_*(B \times \mathbb{R}^d) = \mu(B)$ 
and  ${\pi}_*(\mathbb{R}^d \times B) = \nu(B)$.
Let $\Pi(\mu,\nu)$ be the set of all coupling between $\mu$ and $\nu$.  
For any $\mu,\nu\in \mathcal{P}_p$, the Wasserstein distance of order $p$ between $\mu$ and $\nu$, $\mathcal{W}_p(\mu,\nu)$, is defined by 
\[
\mathcal{W}_p(\mu,\nu):=\inf\limits
\left\{\left(\int_{\mathbb{R}^d \times \mathbb{R}^d}\|x-y\|^p {\pi}_*(\ud x, \ud y)\right)^{\nicefrac{1}{p}}:
{\pi}_* \in \Pi(\mu,\nu)
\right\}.
\]
Let $X$ and $Y$ be two random vectors on $\mathbb{R}^d$ defined on the probability space $(\Omega,\mathcal{F},\mathbb{P})$ with finite $p$-moment. The Wasserstein distance of order $p$ between $X$ and $Y$, $\mathfrak{W}_p(X,Y)$, is defined by
$\mathfrak{W}_p(X,Y):=\mathcal{W}_p(\mathbb{P}_X,\mathbb{P}_Y)$,
where $\mathbb{P}_X$ and $\mathbb{P}_Y$ are  the push-forward probability measures
$\mathbb{P}_X(B):=\mathbb{P}(X\in B)$
and
$\mathbb{P}_Y(B):=\mathbb{P}(Y\in B)$
for any $B\in\mathcal{B}(\mathbb{R}^d)$.
For short, we write 
$\mathcal{W}_p(X,Y)$ in place of $\mathfrak{W}_p(X,Y)$.
A remarkable property that we use along this manuscript is the following scaling property
\begin{equation}\label{eq:scaling}
\mathcal{W}_p(\mathfrak{c}X,\mathfrak{c}Y)=|\mathfrak{c}|\mathcal{W}_p(X,Y)
\quad \textrm{ for any }\quad \mathfrak{c}\in \mathbb{R}.
\end{equation}
The Wasserstein distance metrizes the weak convergence in the space of probabilities with finite $p$-moment. 
It is a fundamental concept in optimal transport theory, probability theory and partial differential equations.
The Wasserstein distance is a natural way  to compare the law  of two random variables  $X$ and $Y$ (even for degenerate cases), where one variable is derived from the other by a small perturbation.
For further details and properties of the Wassertein distance, we refer to the monographies~\cite{PZ} and~\cite{VI}.

\subsection{\textbf{Results}}
We denote by  $\mathcal{N}{\left(v,\Xi \right)}$ the Gaussian distribution in $\mathbb{R}^d$ with vector mean $v$ and positive definite covariance matrix $\Xi$. Let $I_d$ be the identity $d\times d$-matrix. Given a matrix $A\in \mathbb{R}^{d\times d}$, denote by $A^*$ the transpose matrix of $A$ and denote by $\mathrm{Tr}(A)$ the trace of $A$. 

The main result of this manuscript is the following.
\begin{theorem}[Gaussian $\mathcal{W}_2$-approximation of the invariant measure $\mu^\epsilon$]\label{main}
Assume Hypotheses~\eqref{H}, \eqref{C3}, \eqref{eq:lipt} 
and~\eqref{eq:elip} are valid.
Let $\mathcal{J}^{\epsilon}$ be a random vector on $\mathbb{R}^d$ with distribution $\mu^{\epsilon}$. Then there exists a positive constant $K:=K(\delta,\ell,d,c_0,\sigma(0_d))$ 
such that for any $\epsilon\in (0,\epsilon_*)$ with
\[
\epsilon_*=\min\left\{\frac{\delta}{8c_1\|\sigma(0_d)(\sigma(0_d))^*\|_{\mathrm{F}}\cdot d^2},\frac{ \delta}{2\ell^2}\right\}
\]
 it follows that
\begin{equation}\label{eq:resultado}
\mathcal{W}_2\left(\frac{\mathcal{J}^{\epsilon}}{\sqrt{\epsilon}},\mathcal{N}\right)\leq K\sqrt{\epsilon},
\end{equation}
where $\mathcal{N}$ denotes 
the Gaussian distribution on $\mathbb{R}^d$ with zero-mean vector and covariance matrix $\Sigma$ which is the unique solution of the Lyapunov matrix equation
\begin{equation}\label{eq:mleq}
DF(0_d)\Sigma+\Sigma(DF(0_d))^*=\sigma(0_d)(\sigma(0_d))^*.
\end{equation}
\end{theorem}

Using the coupling approach, rates of
convergence of the time evolution to equilibrium in the Wasserstein distance for Langevin processes are given 
in~\cite{Eberle2019} for the underdamped dynamics and 
in~\cite{EberleTAMS2019} for the overdamped dynamics. 
In~\cite{Bolley}, linking functional inequalities with the dissipation to ensure a spectral gap, 
it is shown  that the solution of the  Fokker--Planck equation converges 
in Wasserstein distance of order $2$
to its equilibrium as the time evolution goes by. 
However, the authors in~\cite{Bolley}, \cite{Eberle2019} 
and~\cite{EberleTAMS2019} do not study small random perturbations of dynamical systems, and hence,  an asymptotic analysis for the invariant measure is not needed there.

The proof of Theorem~\ref{main}
does not rely on explicit computations of  $\mu^\epsilon$
neither on explicit formula of the Wasserstein distance of order $2$ between Gaussian distributions. The It\^o formula with the help of~\eqref{H} and~\eqref{eq:lipt} implies that  
 the $p$-moments are bounded recursively as a function of moments of order $p$ and $p-2$. 
Consequently, by an analogous reasoning (but more involved) one can see that the proof of Theorem~\ref{main} can be adapted for any $L^p$-Wasserstein distance for any $p\geq 1$.

\begin{remark}[A comment about total variation convergence for additive noise]
We stress that~\eqref{eq:resultado} does not imply directly any convergence of the corresponding densities.
In other words, the following approximation of densities
\begin{equation}\label{eq:convtv}
\mu^{\epsilon}(\ud x) \approx \epsilon^{-d/2}\mathcal{N}\left(\nicefrac{\ud x}{\sqrt{\epsilon}}\right)\quad \textrm{ for }\quad \epsilon \ll 1.
\end{equation}
cannot be straightforward  deduced from~\eqref{eq:resultado}.
For additive noise, that is, $\sigma(x)=I_d$ for all $x\in \mathbb{R}^d$, using Theorem~5.1, p. 30 in~\cite{Kabanov} (implicitly the celebrated Cameron--Martin--Girsanov Theorem)
it is shown that~\eqref{eq:convtv} is valid, see
Proposition~3.7, p. 1190 in~\cite{BJ1} for further details. However, no rate of convergence is given there.
Multiplicative noise is implicitly discussed in p. 123 
of~\cite{Day1987}.
\end{remark}

\begin{remark}[A word about the constant $K$]
The constant $K$ given in the right-hand side 
of~\eqref{eq:resultado} can be taken as 
\[
K=\frac{96c_0 d^2 \|\sigma(0_d)(\sigma(0_d))^*\|_{\mathrm{F}}}{\delta^2}+\frac{2\ell C^{\nicefrac{1}{2}}_0}{\delta}\quad \textrm{ with }\quad C_0=2\mathrm{Tr}((\sigma(0_d))^*\sigma(0_d)).
\]
We emphasize that the error term $K\sqrt{\epsilon}$ may not be optimal
\end{remark}

\begin{remark}[Existence, uniqueness and integral representation for the covariance matrix $\Sigma$]
By~\eqref{H} and~\eqref{eq:elip}, Theorem~1, p. ~443 
of~\cite{LANTI} implies that~\eqref{eq:mleq} possesses a unique solution. Moreover, Theorem~3, p.~414 of~\cite{LANTI} yields the integral representation for its solution
\[
\Sigma=\int_{0}^\infty
e^{-DF(0_d)s}
\sigma(0_d)(\sigma(0_d))^*
e^{-(DF(0_d))^*s}
\ud s.
\]
\end{remark}

As a consequence of Theorem~\ref{main} we have the following corollaries.
\begin{corollary}[$\mathcal{W}_p$ convergence for \textrm{$p\in [1,2]$}]\label{cor:p}
Assume Hypotheses~\eqref{H}, \eqref{C3}, \eqref{eq:lipt} 
and~\eqref{eq:elip} are valid.
Let $\mathcal{J}^\epsilon $, $\mathcal{N}$, $K$ and $\epsilon_*$ be as in Theorem~\ref{main}.
For any $p\in [1,2]$ it follows that
\[
\mathcal{W}_p\left(\frac{\mathcal{J}^{\epsilon}}{\sqrt{\epsilon}},\mathcal{N}\right)\leq K\sqrt{\epsilon}\quad \textrm{ for all }\quad \epsilon\in (0,\epsilon_*).
\]
\end{corollary}

\begin{proof}
The proof follows immediately by the H\"older inequality 
and~\eqref{eq:resultado}.
\end{proof}

\begin{corollary}[Concentration]\label{cor:conce}
Assume Hypotheses~\eqref{H}, \eqref{C3}, \eqref{eq:lipt} 
and~\eqref{eq:elip} are valid.
Let $\mathcal{J}^\epsilon $ be as in Theorem~\ref{main}.   
For any $p\in [1,2]$ and $\beta<\nicefrac{1}{2}$ it follows that
\[
\lim\limits_{\epsilon \to 0^+}\frac{1}{\epsilon^{\beta}}\mathcal{W}_p(\mathcal{J}^{\epsilon},{0_d})=0.
\]
\end{corollary}

\begin{proof}
The proof follows by the triangle inequality for $\mathcal{W}_p$, Property~\eqref{eq:scaling} and Theorem~\ref{main}.
\end{proof}

The study of the concentration of the equilibrium measure
has been of considerable interest to physicists.
Theorem~1 in~\cite{Biswas2009} implies that $\mathcal{J}^\epsilon\to {0_d}$ as $\epsilon \to 0^+$ in distribution sense. However, it does not say anything about its rate of convergence as Corollary~\ref{cor:conce}.
Results about quantitative
concentration of stationary measures on
attractors and repellers for multiplicative noise are given 
in~\cite{Yingfei} 
and~\cite{YingfeiShen}.

The rest of the manuscript is organized as follows. 
Section~\ref{sec:outline} describes the outline of the proof for the main Theorem~\ref{main}.
Section~\ref{sec:proof} is
devoted to the proofs of the results skipped in 
Section~\ref{sec:outline}.
Finally, in the Appendix~\ref{ap:A}  we provide polynomials and exponential moments estimates for the Ornstein--Uhlenbeck process that we use  in 
Section~\ref{sec:proof}.

\section{\textbf{Outline of the proof}}\label{sec:outline} 

\subsection{\textbf{Linear diffusion approximation}}\label{sub:slda}
Due to the dissipativity condition~\eqref{H},
the nonlinear random dynamics $(X^{\epsilon}_t(x))_{t\geq 0}$ is pushed-back to the origin with high probability. 
In a neighbourhood of the origin, it is reasonable
that an Ornstein--Uhlenbeck process helps us to understand $(X^{\epsilon}_t(x))_{t\geq 0}$ for large times.
Let $(Y_t(x))_{t\geq 0}$ be the unique strong solution of the following linear SDE
\begin{equation}
\label{SEDONOU}
\left\{
\begin{array}{r@{\;=\;}l}
\ud Y^\epsilon_t(x) & - DF(0_d)Y^\epsilon_t(x) \ud t + \sqrt{\epsilon}\sigma(0_d)\ud B_t\quad 
\textrm{ for any }\quad t\geq 0,\\
Y^\epsilon_0(x) & x,
\end{array}
\right.
\end{equation}
where $(B_t)_{t\geq 0}$ is a standard Brownian motion on $\mathbb{R}^d$ and 
$DF(0_d)$ denotes the Jacobian matrix at the point $0_d$. 
The method of variation of parameters  yields
\begin{equation}\label{e:formulita}
\begin{split}
Y^{\epsilon}_t(x)=e^{-DF(0_d)t}x+\sqrt{\epsilon}\, e^{-DF(0_d)t}\int_{0}^{t}e^{DF(0_d)s}\sigma(0_d)\ud B_s\quad \textrm{ for any }\quad t\geq 0.
\end{split}
\end{equation}
Formula~\eqref{e:formulita} implies that
for any $t>0$, $Y^{\epsilon}_t(x)$ possesses Gaussian distribution with vector mean $m_t(x):=e^{-DF(0_d)t}x$ and covariance matrix 
$\Sigma^{\epsilon}_t:=\epsilon\Sigma_t$ for any $t\geq 0$,
where $(\Sigma_t)_{t\geq 0}$ solves the following matrix differential equation
\begin{equation}
\label{EDO111}
\left\{
\begin{array}{r@{\;=\;}l}
\frac{\ud}{ \ud t}\Sigma_t & -DF(0_d)\Sigma_t-\Sigma_t(DF(0_d))^*+\sigma(0_d)(\sigma(0_d))^*\quad \textrm{ for any }\quad t\geq 0,\\
\Sigma_0 & 0_{d\times d},
\end{array}
\right.
\end{equation}
where $0_{d\times d}$ is the $d$-squared zero matrix. We refer to Section~3.7 in~\cite{pavliotisbook} for further details.
By~\eqref{H} one can easily see that 
the eigenvalues of $DF(0_d)$ are contained in the set 
$\{z\in \mathbb{C}: \Re{(z)}\geq \delta\}$.
As a consequence, we have
\[
\|m_t(x)\|\leq e^{-\delta t}\|x\|\to 0 \quad \textrm{ as }\quad t\to \infty.
\]
If in addition, we assume that $\sigma(0_d)$ is invertible,
Lemma~\ref{lem:covariance} in Appendix~\ref{ap:A}  implies
\begin{equation*}
\|\Sigma_t-\Sigma\|_{\mathrm{F}}\leq
\|\Sigma\|^2_{\mathrm{F}} \, e^{-2\delta t}\to 0 \quad \textrm{ as }\quad t\to \infty,
\end{equation*}
where 
$\Sigma$ is the unique solution of the matrix Lyapunov equation~\eqref{eq:mleq}.
Therefore, the limiting distribution of $Y^{\epsilon}_t(x)$ is a Gaussian law with zero-mean vector and positive definite covariance matrix 
$\epsilon\Sigma$.
Moreover, Proposition~3.5 in~\cite{pavliotisbook} implies that $\mathcal{N}(0_d,\epsilon\Sigma)$
is the unique invariant probability measure for the dynamics given by~\eqref{SEDONOU}.
\subsection{\textbf{Disintegration}}
For short we write $\mathcal{N}$
in a place of $\mathcal{N}(0_d,\Sigma)$.
Recall that $\mathcal{J}^{\epsilon}$ denotes a random vector on $\mathbb{R}^d$ with distribution $\mu^{\epsilon}$.
Let $t\geq 0$ and $x_0\in \mathbb{R}^d$. The triangle inequality for the distance $\mathcal{W}_2$ yields
\begin{equation}\label{ine0}
\begin{split}
\mathcal{W}_2\left(\mathcal{J}^{\epsilon},\sqrt{\epsilon}\mathcal{N}\right)\leq  
\mathcal{W}_2\left(\mathcal{J}^{\epsilon},X^{\epsilon}_t(x_0)\right)+\mathcal{W}_2\left(X^{\epsilon}_t(x_0),Y^{\epsilon}_t(x_0)\right)+\mathcal{W}_2\left(Y^{\epsilon}_t(x_0),\sqrt{\epsilon}\mathcal{N}\right).
\end{split}
\end{equation}
Since $\mu^{\epsilon}$ is invariant for the 
dynamics~\eqref{over}, for any $t\geq 0$,
$X^{\epsilon}_t(\mathcal{J}^{\epsilon})$ has distribution $\mu^{\epsilon}$.
By disintegration, the first-term of the right-hand side 
of~\eqref{ine0} can be estimated as follows
\begin{align}\label{ine200}
\mathcal{W}_2\left(\mathcal{J}^{\epsilon},X^{\epsilon}_t(x_0)\right)
\leq \int_{\mathbb{R}^d}\mathcal{W}_2\left(X^{\epsilon}_t(x),X^{\epsilon}_t(x_0)\right)
\mu^{\epsilon}(\ud x).
\end{align}
Analogously,
\begin{align}\label{ine600}
\mathcal{W}_2\left(\sqrt{\epsilon}\mathcal{N},Y^{\epsilon}_t(x_0)\right)
\leq \int_{\mathbb{R}^d}\mathcal{W}_2\left(Y^{\epsilon}_t(x),Y^{\epsilon}_t(x_0)\right)
\mathcal{N}(0_d,\epsilon \Sigma)(\ud x),
\end{align}
where $\mathcal{N}(0_d,\epsilon \Sigma)(\ud x)$ denotes the density of $\sqrt{\epsilon}\mathcal{N}$.
Combining~\eqref{ine0},  
\eqref{ine200} and~\eqref{ine600} we obtain
\begin{align*}
\mathcal{W}_2\left(\mathcal{J}^{\epsilon},\sqrt{\epsilon}\mathcal{N}\right)\leq & \int_{\mathbb{R}^d}\mathcal{W}_2\left(X^{\epsilon}_t(x),X^{\epsilon}_t(x_0)\right)
\mu^{\epsilon}(\ud x)
+\mathcal{W}_2\left(X^{\epsilon}_t(x_0),Y^{\epsilon}_t(x_0)\right)\\
&\quad+\int_{\mathbb{R}^d}\mathcal{W}_2\left(Y^{\epsilon}_t(x),Y^{\epsilon}_t(x_0)\right)
\mathcal{N}(0_d,\epsilon \Sigma)(\ud x)
\end{align*}
for any $t\geq 0$ and $x_0\in \mathbb{R}^d$.
In particular,  for any $t\geq 0$ we have
\begin{equation}\label{formulainequality}
\begin{split}
\mathcal{W}_2\left(\mathcal{J}^{\epsilon},\sqrt{\epsilon}\mathcal{N}\right)\leq & \int_{\mathbb{R}^d}\mathcal{W}_2\left(X^{\epsilon}_t(x),X^{\epsilon}_t(0_d)\right)
\mu^{\epsilon}(\ud x)
+\mathcal{W}_2\left(X^{\epsilon}_t(0_d),Y^{\epsilon}_t(0_d)\right)\\
&\quad+\int_{\mathbb{R}^d}\mathcal{W}_2\left(Y^{\epsilon}_t(x),Y^{\epsilon}_t(0_d)\right)
\mathcal{N}(0_d,\epsilon \Sigma)(\ud x).
\end{split}
\end{equation}
In what follows, we provide the tools for estimating the right-hand side of~\eqref{formulainequality}.
The following lemma allows us to couple two solutions 
of~\eqref{over} starting in  different initial conditions.
\begin{lemma}[Synchronous coupling I]\label{lemmaA}
Assume Hypotheses~\eqref{H} and~\eqref{eq:lipt} are valid.
Let $x,x_0\in \mathbb{R}^d$. Then
\[
\mathcal{W}_2\left(X^{\epsilon}_t(x),X^{\epsilon}_t(x_0)\right)\leq e^{-(\nicefrac{\delta}{2}) t}\|x-x_0\|\quad \textrm{ for all }\quad t\geq 0,\,\epsilon \in (0,\nicefrac{\delta}{\ell^2}],
\]
where $\delta>0$ is the dissipativity constant that appears 
in~\eqref{H} and $\ell$ is the Lipschitz constant that appears in~\eqref{eq:lipt}.
In particular, 
\[
\mathcal{W}_2\left(X^{\epsilon}_t(x),X^{\epsilon}_t(0_d)\right)\leq e^{-(\nicefrac{\delta}{2}) t}\|x\|\quad \textrm{ for all }\quad t\geq 0,\,\epsilon \in (0,\nicefrac{\delta}{\ell^2}].
\]
\end{lemma}
The following lemma provides second moment estimates for the marginals of the process~\eqref{over} and also for its invariant probability measure $\mu^\epsilon$.
\begin{lemma}[Second moment estimates]\label{lemmaB}
Assume Hypotheses~\eqref{H} and~\eqref{eq:lipt} are valid.
For any $x\in \mathbb{R}^d$ we have
\[
\mathbb{E}[\|X^{\epsilon}_t(x)\|^2]\leq 
\|x\|^2 e^{-\delta t}+\frac{ \epsilon C_0}{\delta}\quad \textrm{ for all }\quad t\geq 0,\,\epsilon\in (0,\nicefrac{\delta}{(2\ell^2})],
\]
where $\delta>0$ is the dissipativity constant that appears 
in~\eqref{H}, $\ell$ is the Lipschitz constant that appears 
in~\eqref{eq:lipt}
and $C_0=2\mathrm{Tr}((\sigma(0_d))^*\sigma(0_d))$.
In addition,
\begin{equation}\label{eq:limmomen}
\int_{\mathbb{R}^d}\|x\|^2 \mu^{\epsilon}(\ud x)\leq \frac{\epsilon C_0}{\delta}
\quad \textrm{ for all }\quad \epsilon\in (0,\nicefrac{\delta}{(2\ell^2})].
\end{equation}
\end{lemma}
The next lemma is crucial in our argument. Due to the contracting nature of the dynamics, the random dynamics around zero, 
$(X^\epsilon_t(0_d))_{t\geq 0}$, can be approximated from its linearization $(Y^\epsilon_t(0_d))_{t\geq 0}$.
\begin{lemma}[Synchronous coupling II]\label{lemmaC}
Assume Hypotheses~\eqref{H}, \eqref{C3}, \eqref{eq:lipt} 
and~\eqref{eq:elip} are valid.
Then there exists a positive constant $C:=C(\delta,\ell, d,c_0,\sigma(0_d))$ such that for any $\epsilon\in (0,\epsilon_*)$ with
\[
\epsilon_*:=\min\left\{\frac{\delta}{8c_1\|\sigma(0_d)(\sigma(0_d))^*\|_{\mathrm{F}}\cdot d^2},\frac{ \delta}{2\ell^2}\right\},
\]
 and for all $t\geq 0$ we have
\[
\mathcal{W}_2\left(X^{\epsilon}_t(0_d),Y^{\epsilon}_t(0_d)\right)\leq 
 C \epsilon,
\]
where $\delta>0$ is the dissipativity constant that appears 
in~\eqref{H},
$c_0$, $c_1$ are the positive constants that appear 
in~\eqref{C3} 
and
$\ell$ is the Lipschitz constant that appears 
in~\eqref{eq:lipt}.
\end{lemma}
We point out that the constant $C$ can be taken as 
\[
C=\frac{48c_0 d^2 \|\sigma(0_d)(\sigma(0_d))^*\|_{\mathrm{F}}}{\delta^2}+\frac{\ell C^{\nicefrac{1}{2}}_0}{\delta}.
\]
The latter is deduced from~\eqref{eq:contc}.
Recall that $\Sigma$ is the solution of~\eqref{EDO111}.
Since for any $t\geq 0$, $Y^{\epsilon}_t(\mathcal{N}(0_d,\epsilon\Sigma_t))$ has distribution  $\mathcal{N}(0_d,\epsilon\Sigma)$, an analogous reasoning used in the proofs of Lemma~\ref{lemmaA} and Lemma~\ref{lemmaB} implies the following lemma.
\begin{lemma}[Synchronous coupling III]\label{lemmaD}
Assume Hypotheses~\eqref{H} and~\eqref{eq:lipt} are valid.
For any $x\in \mathbb{R}^d$ it follows that
\[
\mathcal{W}_2\left(Y^{\epsilon}_t(x),Y^{\epsilon}_t(0_d)\right)\leq e^{-(\nicefrac{\delta}{2}) t}\|x\|\quad \textrm{ for all }\quad t\geq 0,\,\epsilon \in (0,\nicefrac{\delta}{\ell^2}],
\]
where $\delta>0$ is the dissipativity constant that appears 
in~\eqref{H} and $\ell$ is the Lipschitz constant that appears in~\eqref{eq:lipt}.
In addition, assume that $\sigma(0_d)$ is invertible. Then it follows that
\begin{equation}\label{e:normalint}
\int_{\mathbb{R}^d}\|x\|^2 \mathcal{N}(0_d,\epsilon \Sigma)(\ud x)\leq  \epsilon d \|\Sigma^{\nicefrac{1}{2}}\|^2_{\mathrm{F}}.
\end{equation}
\end{lemma}
For simplicity we assume that $\sigma(0_d)$ is invertible in Lemma~\ref{lemmaD}. Actually, it is not needed to obtain an estimate such as~\eqref{e:normalint}. Nevertheless, it is enforced to define the so-called generalized Gaussian distribution with degenerate covariance matrix and hence the notion of Moore--Penrose pseudoinverse is required.
The assumption that $\sigma(0_d)$ is invertible can be removed and~\eqref{eq:limmomen}  in Lemma~\ref{lemmaD} remains valid replacing $\mu^\epsilon$ by the law of $\mathcal{N}(0_d,\epsilon \Sigma)$.

In the sequel, we stress the fact that 
Theorem~\ref{main} is just a consequence of what we have already stated up to here. 
\begin{proof}[Theorem~\ref{main}]
By~\eqref{formulainequality}, Lemma~\ref{lemmaA}, Lemma~\ref{lemmaB}, Lemma~\ref{lemmaC} and Lemma~\ref{lemmaD} we have
\begin{equation}\label{eq:inecomplete}
\mathcal{W}_2\left(\mathcal{J}^{\epsilon},\sqrt{\epsilon}\mathcal{N}\right)\leq  
\sqrt{\frac{\epsilon C_0}{\delta}}e^{-(\nicefrac{\delta}{2}) t}
+C\epsilon
+
\sqrt{\epsilon d \|\Sigma^{\nicefrac{1}{2}}\|^2_{\mathrm{F}}}e^{-(\nicefrac{\delta}{2}) t}
\end{equation}
for any $t\geq 0$ and $\epsilon\in (0,\epsilon_*]$.
Due to~\eqref{eq:scaling},  \eqref{eq:inecomplete} implies
\begin{equation*}
\mathcal{W}_2\left(\frac{\mathcal{J}^{\epsilon}}{\sqrt{\epsilon}},\mathcal{N}\right)\leq  \sqrt{\frac{ C_0}{\delta}}e^{-(\nicefrac{\delta}{2}) t}
+C\sqrt{\epsilon}+
\sqrt{ d \|\Sigma^{\nicefrac{1}{2}}\|^2_{\mathrm{F}}}e^{-(\nicefrac{\delta}{2}) t}
\end{equation*}
for any $t \geq 0$ and $\epsilon\in (0,\epsilon_*]$.
The cunning choice 
\[
t_\epsilon=\max\left\{\frac{1}{\delta}\ln\left(\frac{4C_0}{\delta C^2\epsilon}\right),\frac{1}{\delta}\ln\left(\frac{4d \|\Sigma^{\nicefrac{1}{2}}\|^2_{\mathrm{F}}}{ C^2\epsilon}\right)\right\}
\]
 yields
\[
\mathcal{W}_2\left(\frac{\mathcal{J}^{\epsilon}}{\sqrt{\epsilon}},\mathcal{N}\right)\leq 2C\sqrt{\epsilon},
\]
which concludes Theorem~\ref{main}.
\end{proof}
\section{\textbf{Proofs}}\label{sec:proof}
In this section, we give the proofs of Lemma~\ref{lemmaA}, Lemma~\ref{lemmaB} and Lemma~\ref{lemmaC}.
Along their proofs, we use several times
the celebrated Gr\"onwall inequality.
We state it here as a lemma for the sake of completeness.
\begin{lemma}[Gr\"onwall's inequality]\label{lem:gron} Let $T>0$ be fixed, $\mathfrak{g}:[0,T]\rightarrow \mathbb{R}$
be a $\mathcal{C}^1$-function and 
$\mathfrak{h}:[0,T]\rightarrow \mathbb{R}$ be a $\mathcal{C}^0$-function. Assume that
\[
\frac{\ud}{\ud t} \mathfrak{g}(t)\leq -a  
\mathfrak{g}(t) + \mathfrak{h}(t)\quad 
\textrm{ for any }\quad t \in [0, T],
\]
where $a\in\mathbb{R}$, and the derivative at $0$ and $T$ are understood as the right
and left derivatives, respectively. 
Then
\[
\mathfrak{g}(t)\leq e^{-at}\mathfrak{g}(0) + e^{-at}\int_{0}^{t}
e^{as}\mathfrak{h}(s)\ud s\quad \textrm{ for any }\quad t\in [0, T].
\]
\end{lemma}
\subsection{\textbf{The synchronous coupling I}}
For any $x,x_0\in \mathbb{R}^d$, let 
$(X^{\epsilon}_t(x))_{t\geq 0}$ and $(X^{\epsilon}_t(x_0))_{t\geq 0}$ be the solutions of~\eqref{over} with initial conditions $x$ and $x_0$, respectively.
In the sequel, we consider the so-called synchronous coupling, i.e., both processes $(X^{\epsilon}_t(x))_{t\geq 0}$ and $(X^{\epsilon}_t(x_0))_{t\geq 0}$ have  the same  driving noise $(B_t)_{t\geq 0}$.
\begin{proof}[Lemma~\ref{lemmaA}]
By the It\^o formula we have
\begin{align*}
\ud \|X^{\epsilon}_t(x)-& X^{\epsilon}_t(x_0)\|^2=-2\<X^{\epsilon}_t(x)-X^{\epsilon}_t(x_0),F(X^{\epsilon}_t(x)) -F(X^{\epsilon}_t(x_0))\>\ud t\\
&\hspace{3cm}+
\epsilon\,\mathrm{Tr}[(\sigma(X^{\epsilon}_t(x))-\sigma(X^{\epsilon}_t(x_0)))^* (\sigma(X^{\epsilon}_t(x))-\sigma(X^{\epsilon}_t(x_0)))]\ud t
\\
&\hspace{3cm}+2\sqrt{\epsilon}\<X^{\epsilon}_t(x)-X^{\epsilon}_t(x_0),(\sigma(X^{\epsilon}_t(x))-\sigma(X^{\epsilon}_t(x_0)))\ud B_t\>.
\end{align*}
By~\eqref{eq:lipt} we have
\begin{equation}\label{eq:traze1}
\begin{split}
\mathrm{Tr}[(\sigma(X^{\epsilon}_t(x))- \sigma(X^{\epsilon}_t(x_0)))^* (\sigma(X^{\epsilon}_t(x))-\sigma(X^{\epsilon}_t(x_0)))]\leq \ell^2  \|X^{\epsilon}_t(x)-X^{\epsilon}_t(x_0)\|^2.
\end{split}
\end{equation}
A localization argument with the help of~\eqref{H} 
and~\eqref{eq:traze1}
implies 
\begin{align*}
\frac{\ud}{\ud t} \mathbb{E}[\|X^{\epsilon}_t(x)-X^{\epsilon}_t(x_0)\|^2] &\leq -2\delta \mathbb{E}\left[\|X^{\epsilon}_t(x)-X^{\epsilon}_t(x_0)\|^2\right]+ \epsilon  \ell^2  \, \mathbb{E}\left[\|X^{\epsilon}_t(x)-X^{\epsilon}_t(x_0)\|^2\right]\\
&\leq -(2\delta-\epsilon  \ell^2)\mathbb{E}\left[\|X^{\epsilon}_t(x)-X^{\epsilon}_t(x_0)\|^2\right]
\end{align*}
for all $t\geq 0$.
Since $\mathbb{E}[\|X^{\epsilon}_0(x)-X^{\epsilon}_0(x_0)\|^2]=\|x-x_0\|^2$, Lemma~\ref{lem:gron} yields
\[
\mathbb{E}[\|X^{\epsilon}_t(x)-X^{\epsilon}_t(x_0)\|^2]\leq e^{-(2\delta-\epsilon  \ell^2) t}\|x-x_0\|^2\quad \textrm{  for  any }\quad t\geq 0.
\]
Therefore, for any
$\epsilon\in (0,\nicefrac{\delta}{\ell^2}]$
we have 
\[
\mathcal{W}_2(X^{\epsilon}_t(x),X^{\epsilon}_t(x_0))
\leq e^{-(\nicefrac{\delta}{2}) t}\|x-x_0\|
\quad \textrm{ for any }\quad
x,x_0\in \mathbb{R}^d,\; t\geq 0.
\]
\end{proof}
\subsection{\textbf{Second moment estimates}}
For any $x\in \mathbb{R}^d$, let 
$(X^{\epsilon}_t(x))_{t\geq 0}$ be the solution 
of~\eqref{over} with initial condition $x$.
\begin{proof}[Lemma~\ref{lemmaB}]
In the sequel, we  estimate 
$\mathbb{E}[\|X^\epsilon_t(x)\|^2]$. 
The It\^o formula
and~\eqref{H} yield
\begin{align*}
\ud \|X^\epsilon_t(x)\|^2
		&= -2 \<X^\epsilon_t(x), F(X^\epsilon_t(x))\>\ud t+  \epsilon\,\mathrm{Tr}[(\sigma(X^{\epsilon}_t(x)))^*
\sigma(X^{\epsilon}_t(x))]+M^\epsilon_t(x) \nonumber\\
		&\leq -2\delta \|X^\epsilon_t(x)\|^2 \ud t +
\epsilon\,\mathrm{Tr}[(\sigma(X^{\epsilon}_t(x)))^*
\sigma(X^{\epsilon}_t(x))]		
		 \ud t
		+M^\epsilon_t(x),
\end{align*}
where $M^\epsilon_t(x):=\<2\sqrt{\epsilon}X^\epsilon_t(x),\ud B_t\>$ for every $t\geq 0$.
Since
\begin{equation*}
\begin{split}
\mathrm{Tr}[(\sigma(X^{\epsilon}_t(x)))^*
\sigma(X^{\epsilon}_t(x))]\leq 2
\mathrm{Tr}[(\sigma(X^{\epsilon}_t(x))-\sigma(0_d))^*(
\sigma(X^{\epsilon}_t(x))-\sigma(0_d))]+2\mathrm{Tr}((\sigma(0_d))^*\sigma(0_d)),
\end{split}
\end{equation*}
Hypothesis~\eqref{eq:lipt} implies
\begin{equation}\label{eq:traza2}
\begin{split}
\mathrm{Tr}[(\sigma(X^{\epsilon}_t(x)))^*
\sigma(X^{\epsilon}_t(x))]\leq 
2\ell^2 \|X^{\epsilon}_t(x) \|^2+C_0,
\end{split}
\end{equation}
where $C_0:=2\mathrm{Tr}((\sigma(0_d))^*\sigma(0_d))$.
 A localization argument with the help of~\eqref{H} 
 and~\eqref{eq:traza2}
implies
\begin{align*}
\frac{\ud}{\ud t}\mathbb{E}[\|X^\epsilon_t(x)\|^2]
	\leq -(2\delta-2\epsilon \ell^2 ) \mathbb{E}[\|X^\epsilon_t(x)\|^2] + \epsilon C_0 \quad \textrm{ for any }\quad t\geq 0.
\end{align*}
Since $\mathbb{E}[\|X^\epsilon_0(x)\|^2]=\|x\|^2$,
for any
$\epsilon\in (0,\nicefrac{\delta}{(2\ell^2})]$
Lemma~\ref{lem:gron}
yields
\begin{equation}\label{e:e1}
\mathbb{E}[\|X^\epsilon_t(x)\|^2]\leq e^{-\delta t}\|x\|^2 +
\frac{\epsilon C_0}{\delta}(1-e^{-\delta t})\leq e^{-\delta t}\|x\|^2 +\frac{\epsilon C_0}{\delta}
\end{equation}
for any $t\geq 0$ and  $x\in \mathbb{R}^d$.
Following the same reasoning used on  p.~39 in~\cite{BJ1},
it is not hard to see that~\eqref{e:e1} implies 
\begin{equation*}
\int_{\mathbb{R}^d} \|x\|^2\mu^{\epsilon}(\ud x)\leq \frac{\epsilon C_0}{\delta}\quad \textrm{ for all }\quad \epsilon\in (0,\nicefrac{\delta}{(2\ell^2})]. 
\end{equation*}
\end{proof}
\subsection{\textbf{The synchronous coupling II}}\label{sec: natcouII}
We consider the  
solution of~\eqref{over} with initial condition $x=0_d$,
$(X^{\epsilon}_t(0_d))_{t\geq 0}$. 
Let $(Y^{\epsilon}_t(0_d))_{t\geq 0}$ be as~\eqref{SEDONOU}. In this section, we use the synchronous coupling between 
$X^{\epsilon}_t(0_d)$ and $Y^{\epsilon}_t(0_d)$, i.e., both processes $(X^{\epsilon}_t(0_d))_{t\geq 0}$ and $(Y^{\epsilon}_t(0_d))_{t\geq 0}$ have  the same  driving noise $(B_t)_{t\geq 0}$.
\begin{proof}[Lemma~\ref{lemmaC}]
In the sequel, we estimate
$\mathbb{E}\left[\|X^{\epsilon}_t(0_d)-Y^{\epsilon}_t(0_d)\|^2\right]$.
Note that $X^{\epsilon}_0(0_d)=Y^{\epsilon}_0(0_d)=0_d$.
Let $\Delta^{\epsilon}_t(0_d):=X^{\epsilon}_t(0_d)-Y^{\epsilon}_t(0_d)$, $t\geq 0$. Then
\[
\begin{split}
\ud \Delta^{\epsilon}_t(0_d)
&=-\left[F(X^{\epsilon}_t(0_d))-F(Y^{\epsilon}_t(0_d))\right]\ud t
+\left[DF(0_d)Y^{\epsilon}_t(0_d)-F(Y^{\epsilon}_t(0_d))\right]\ud t\\
&\quad\quad \quad+
\sqrt{\epsilon}( \sigma(X^{\epsilon}_t(0_d))
-\sigma(0_d)
)
\ud B_t.
\end{split}
\]
Hence, the It\^o formula reads
\[
\begin{split}
\ud \|\Delta^{\epsilon}_t(0_d)\|^2
&=-2\<\Delta^{\epsilon}_t(0_d),F(X^{\epsilon}_t(0_d))-F(Y^{\epsilon}_t(0_d))\>\ud t\\
&\quad\quad \quad+2\<\Delta^{\epsilon}_t(0_d),DF(0_d)Y^{\epsilon}_t(0_d)-F(Y^{\epsilon}_t(0_d))\>\ud t
\\
&\quad\quad \quad +\epsilon\, \mathrm{Tr}
[
(\sigma(X^{\epsilon}_t(0_d))
-\sigma(0_d))^*( \sigma(X^{\epsilon}_t(0_d))
-\sigma(0_d))
]\ud t \\
&\quad\quad \quad +2\sqrt{\epsilon}
\<\Delta^{\epsilon}_t(0_d),( \sigma(X^{\epsilon}_t(0_d))
-\sigma(0_d)
)
\ud B_t\>.
\end{split}
\]
By~\eqref{eq:lipt} we have 
\begin{equation}\label{eq:trazita}
\mathrm{Tr}[
(\sigma(X^{\epsilon}_t(0_d))
-\sigma(0_d))^*( \sigma(X^{\epsilon}_t(0_d))
-\sigma(0_d)]\leq \ell^2 \|X^{\epsilon}_t(0_d)\|^2.
\end{equation}
A localization argument with the help of~\eqref{H},
the Cauchy--Schwarz inequality
and~\eqref{eq:trazita}
implies
\begin{equation}\label{e:difeq}
\begin{split}
\frac{\ud}{\ud t} 
\mathbb{E}[\|\Delta^{\epsilon}_t&(0_d)\|^2] \leq -2\delta \mathbb{E}[\|\Delta^{\epsilon}_t(0_d)\|^2]\\
&\quad\quad \quad\quad \quad +2
\mathbb{E}[\|\Delta^{\epsilon}_t(0_d)\|\cdot
\|F(Y^{\epsilon}_t(0_d))-DF(0_d)Y^{\epsilon}_t(0_d)\|]+\epsilon \ell^2 \mathbb{E}[\|X^{\epsilon}_t(0_d)\|^2].
\end{split}
\end{equation}
Differential inequality~\eqref{e:difeq} and the Young inequality (for $p=2$) yield
\[
\begin{split}
\frac{\ud}{\ud t} 
\mathbb{E}[\|\Delta^{\epsilon}_t(0_d)\|^2]\leq -\delta \mathbb{E}[\|\Delta^{\epsilon}_t(0_d)\|^2]+\frac{1}{\delta}
\mathbb{E}[\|F(Y^{\epsilon}_t(0_d))-DF(0_d)Y^{\epsilon}_t(0_d)\|^2]+\epsilon \ell^2 \mathbb{E}[\|X^{\epsilon}_t(0_d)\|^2].
\end{split}
\]
By Lemma~\ref{lemmaB} we have
\[
\mathbb{E}[\|X^{\epsilon}_t(0_d)\|^2]\leq 
\frac{ \epsilon C_0}{\delta}\quad \textrm{ for all }\quad t\geq 0,\,\epsilon \in (0,\nicefrac{\delta}{(2\ell^2)}],
\]
where $C_0=2\mathrm{Tr}((\sigma(0_d))^*\sigma(0_d))$.
Since $\Delta^{\epsilon}_t(0_d)=0$, Lemma~\ref{lem:gron} implies
\begin{equation}
\begin{split}
\mathbb{E}[\|\Delta^{\epsilon}_t(0_d)\|^2]&\leq \frac{1}{\delta}e^{-\delta t}\int_{0}^{t}
e^{\delta s}\mathbb{E}[\|F(Y^{\epsilon}_s(0_d))-DF(0_d)Y^{\epsilon}_s(0_d)\|^2]\ud s
+
\frac{ \epsilon^2 \ell^2 C_0}{\delta^2}
\\
& \leq \frac{1}{\delta^2}\sup\limits_{0\leq s\leq t}\mathbb{E}[\|F(Y^{\epsilon}_s(0_d))-DF(0_d)Y^{\epsilon}_s(0_d)\|^2]
+
\frac{ \epsilon^2 \ell^2 C_0}{\delta^2}
\label{eq: es0}
\end{split}
\end{equation}
for all  $t\geq 0$ and $\epsilon \in (0,\nicefrac{\delta}{(2\ell^2)}]$.
Next, we estimate 
\[
\sup\limits_{0\leq s\leq t}\mathbb{E}[\|F(Y^{\epsilon}_s(0_d))-DF(0_d)Y^{\epsilon}_s(0_d)\|^2].
\]
Let $s\in [0,t]$. 
Recall that $F\in \mathcal{C}^2(\mathbb{R}^d,\mathbb{R}^d)$. Since $F(0_d)=0_d$, The mean value theorem yields
\[
\begin{split}
F(Y^{\epsilon}_s(0_d))-F(0_d)
=
\int_{0}^{1} DF(\theta_1 Y^{\epsilon}_s(0_d))\ud \theta_1 \cdot Y^{\epsilon}_s(0_d),
\end{split}
\]
where $DF$ denotes the derivative of $F$.
Since $F(0_d)=0_d$, we have
\begin{equation}\label{e:mena}
\begin{split}
&F(Y^{\epsilon}_s(0_d))-
DF(0_d) Y^{\epsilon}_s(0_d)=
\int_{0}^{1} \left[DF(\theta_1 Y^{\epsilon}_s(0_d))-DF(0_d)\right]\ud \theta_1\cdot Y^{\epsilon}_s(0_d).
\end{split}
\end{equation}
Applying The mean value theorem 
to~\eqref{e:mena} we deduce
\begin{equation}\label{ine6}
\begin{split}
\|F(Y^{\epsilon}_s(0_d))-
DF(0_d)Y^{\epsilon}_s(0_d))\|
\leq C^{\epsilon}_{s} \|Y^{\epsilon}_s(0_d)\|^2,
\end{split}
\end{equation}
where 
\[
C^{\epsilon}_{s}:=\int_{0}^{1}\int_{0}^{1}\|D^2 F(\theta_1\theta_2 Y^{\epsilon}_s(0_d))\| \ud \theta_1 \ud \theta_2
\]
and $D^2 F$ denotes the second order derivative of $F$.
Note that 
\begin{equation}\label{eq:resc}
Y^{\epsilon}_t(0_d)=\sqrt{\epsilon}\, Y_t\quad \textrm{ for any }\quad t\geq 0, 
\end{equation}
where $(Y_t)_{t\geq 0}$ is the unique strong solution of 
\begin{equation}\label{eq:Ynw}
\left\{
\begin{array}{r@{\;=\;}l}
\ud Y_t&-DF(0_d)Y_t\ud t+\sigma(0_d)\ud B_t\quad  
\textrm{  for any }\quad t\geq 0,\\
Y_0 & 0_d.
\end{array}
\right.
\end{equation}
By~\eqref{eq:resc} and~\eqref{C3} we have
\[
\begin{split}
\|D^2 F(\theta_1\theta_2 Y^{\epsilon}_s(0_d))\|=
\|D^2 F(\theta_1\theta_2 \sqrt{\epsilon}Y_s)\|&\leq c_0e^{c_1\theta^2_1\theta^2_2 \epsilon \|Y_s\|^2}.
\end{split}
\]
Since $\theta_1,\theta_2\in [0,1]$, we obtain
\begin{equation}\label{e:D2ine}
\|D^2 F(\theta_1\theta_2 Y^{\epsilon}_s(0_d))\|\leq c_0e^{c_1\epsilon \|Y_s\|^2}.
\end{equation}
Inequality~\eqref{e:D2ine} with the help of  
inequality~\eqref{ine6} and equality~\eqref{eq:resc} yields
\begin{equation}\label{eq:nomet}
\begin{split}
&\|F(Y^{\epsilon}_s(0_d))-
DF(0_d) Y^{\epsilon}_s(0_d)\|^2
\leq c^2_0e^{2c_1\epsilon \|Y_s\|^2}\epsilon^2\|Y_s\|^4
\end{split}
\end{equation}
for any $s\geq 0$, where $(Y_t)_{t\geq 0}$ is the solution 
of~\eqref{eq:Ynw}.
By item i) of Lemma~\ref{A1} in  Appendix~\ref{ap:A} 
it follows that
\begin{equation}\label{eq:notmet1}
\mathbb{E}[\|Y_s\|^8]\leq 24 C^4_*\quad \textrm{ for any }\quad s\geq 0,
\end{equation}
where 
\begin{equation}\label{e:Cstar}
C_*=\frac{\|\sigma(0_d)(\sigma(0_d))^*\|_{\mathrm{F}}\cdot d^2}{\delta}
\end{equation}
and $\|\cdot\|_{\mathrm{F}}$ denotes the Frobenius norm.
Due to~\eqref{eq:elip}, we note that $C_*>0$.
Moreover, by item ii) Lemma~\ref{A1} in  
Appendix~\ref{ap:A} 
for $\epsilon\in (0,\frac{1}{4c_1C_*})$ we have
\begin{equation}\label{eq:nomet2}
\mathbb{E}[e^{4 c_1\epsilon\|Y_s\|^2}]\leq \frac{1}{1-4\epsilon c_1 C_*}\quad \textrm{ for any }\quad s\geq 0.
\end{equation}
Estimate~\eqref{eq:nomet} with the help of
the Cauchy--Schwarz inequality, \eqref{eq:notmet1} 
and~\eqref{eq:nomet2} implies 
\[
\begin{split}
  \mathbb{E}[\|F(Y^{\epsilon}_s(0_d))-
DF(0_d)Y^{\epsilon}_s(0_d)\|^2]\leq \epsilon^2 c^2_0
\left(\mathbb{E}[e^{4 c_1\epsilon\|Y_s\|^2}]
 \mathbb{E}[\|Y_s\|^8]\right)^{\nicefrac{1}{2}} \tilde{C}(\delta,d,c_0)\epsilon^2\left(\frac{1}{1-4\epsilon c_1 C_*} \right)^{\nicefrac{1}{2}}
\end{split}
\]
for any $s\geq 0$,
$\epsilon\in (0,\frac{1}{4c_1C_*})$,
 where $\tilde{C}(\delta, d,c_0)=\sqrt{24}c^2_0C^2_*$ is a positive constant.
Consequently, for $\epsilon\in (0,\frac{1}{4c_1C})$ we obtain
\begin{align}\label{eq:utu}
\sup\limits_{0\leq s\leq t} \mathbb{E}[\|F(Y^{\epsilon}_s(0_d))-
DF(0_d)Y^{\epsilon}_s(0_d)\|^2]
\leq \tilde{C}(\delta,d,c_0)\epsilon^2\left(\frac{1}{1-4\epsilon c_1 C_*} \right)^{\nicefrac{1}{2}}.
\end{align}
Note that if $\epsilon\in(0,\frac{1}{8c_1C_*})$, then 
$(1-4\epsilon c_1 C_*)\geq \nicefrac{1}{2}$.
Let $\epsilon_*:=\min\big\{\frac{1}{8c_1C_*},\frac{ \delta}{2\ell^2}\big\}$.
By~\eqref{eq: es0} and~\eqref{eq:utu} we have for all
$\epsilon\in (0,\epsilon_*]$ and all $t\geq 0$
\[
\mathbb{E}\left[\|X^{\epsilon}_t(0_d)-Y^{\epsilon}_t(0_d)\|^2\right]\leq 
\frac{\sqrt{2}}{\delta^2}\tilde{C}(\delta,d,c_0)\epsilon^2+\frac{ \epsilon^2 \ell^2 C_0}{\delta^2}.
\]
As a consequence, for any  $\epsilon\in (0,\epsilon_*]$ and  $t\geq 0$
we have
\begin{align}\label{eq:contc}
\mathcal{W}_2\left(X^{\epsilon}_t(0_d),Y^{\epsilon}_t(0_d)\right)&\leq 
\frac{\epsilon}{\delta}(\sqrt{2}\tilde{C}(\delta,d,c_0)+\ell^2 C_0 )^{\nicefrac{1}{2}}\leq 
\frac{\epsilon}{\delta}(48c_0 C_*+\ell C^{\nicefrac{1}{2}}_0 ),
\end{align}
where in the last inequality we use
the subadditivity property of the root-map. 
Inequality~\eqref{eq:contc} with the help of~\eqref{e:Cstar} implies the statement.
\end{proof}

\appendix
\section{Tools}\label{ap:A}
In this section, we compute the even moments and exponential moments of the Ornstein--Uhlenbeck process starting at zero.
Let $(Z_t)_{t\geq 0}$ be the unique strong solution of the linear SDE 
\begin{equation}\label{eq:Z990}
\left\{
\begin{array}{r@{\;=\;}l}
\ud Z_t&-UZ_t\ud t+V\ud B_t\quad  
\textrm{  for any }\quad t\geq 0,\\
Z_0 & 0_d,
\end{array}
\right.
\end{equation}
where $U,V\in \mathbb{R}^{d\times d}$ 
are given matrices. The drift matrix $U$ satisfies the following condition:
 there exists a positive $\delta$ such that
\begin{equation}\label{eq:matrix}
\<Ux,x\>\geq \delta \|x\|^2\quad \textrm{ for all }\quad x\in \mathbb{R}^d.
\end{equation}
We recall the definitions and properties of 1-norm $\|\cdot\|_{1}$ and the Frobenius norm $\|\cdot\|_{\mathrm{F}}$. For a given matrix $A=(a^{i,j})_{i,j\in \{1,\ldots,d\}}$  they are given by
\[
\|A\|_{1}:=\sum\limits_{i,j=1}^{d}|a^{i,j}|\quad \textrm{ and } \quad \|A\|_{\mathrm{F}}:=\sqrt{\sum\limits_{j=1}^{d}|a^{i,j}|^2}.
\]

\begin{lemma}[Polynomial and exponential moments]\label{A1}
Assume that~\eqref{eq:matrix} is valid and let $(Z_t)_{t\geq 0}$ be the unique strong solution of the SDE~\eqref{eq:Z990}. Then the following holds.
\begin{itemize}
\item[i)]
For each $j\in \mathbb{N}$ it follows that
\begin{equation}\label{eq:supremum}
\mathbb{E}[\|Z_t\|^{2j}]\leq C^j_* j!\quad \textrm{ for all }\quad t\geq 0, \quad 
\textrm{ where }\quad
C_*:=\frac{\|VV^*\|_{\mathrm{F}}\cdot d^2}{\delta}.
\end{equation}
\item[ii)]
Let $C_*$ be as in item i).
For any $\lambda \in (0,1/C_*)$ and all $t\geq 0$ it follows that
\begin{equation*}
\mathbb{E}[e^{\lambda \|Z_t\|^2}]\leq \frac{1}{1-\lambda C_*}.
\end{equation*}
\end{itemize}
\end{lemma}

\begin{proof}
We start with the proof of item i).
The proof is done by the induction method.
We start the induction basis, $j=1$.
The
It\^o formula yields
\begin{equation}\label{eq:Z32}
\ud \|Z_t\|^2=-2\<Z_t,UZ_t\>\ud t+\mathrm{Tr}[V^*V]\ud t+2\<Z_t,V \ud B_t\>.
\end{equation}
A localization argument in~\eqref{eq:Z32} with the help 
of~\eqref{eq:matrix} implies
\[
\frac{\ud }{\ud t} \mathbb{E}[\|Z_t\|^2]\leq -2\delta \mathbb{E}[\|Z_t\|^2]+\mathrm{Tr}[V^*V].
\]
Since $Z_t=0_d$,
Lemma~\ref{lem:gron} yields 
\begin{equation}\label{eq:Z2}
 \mathbb{E}[\|Z_t\|^2]\leq \frac{\mathrm{Tr}[V^*V]}{2\delta}\quad \textrm{ for all }\quad t\geq 0.
\end{equation}
Note that
\begin{align}\label{eq:hest}
\frac{\mathrm{Tr}[V^*V]}{2\delta}\leq \frac{\|VV^*\|_{1}}{2\delta}\leq 
\frac{d\, \|VV^*\|_{\mathrm{F}}}{2\delta}\leq 
\frac{d^2\, \|VV^*\|_{\mathrm{F}}}{\delta}.
\end{align}
Combining~\eqref{eq:Z2} and~\eqref{eq:hest} we prove the induction basis.

We assume that~\eqref{eq:supremum} holds for $j=n$ and we  prove that it remains valid for $j=n+1$. 
The It\^o formula for the function $f(x)=\|x\|^{2(n+1)}$, $x\in \mathbb{R}$,
reads
\begin{equation}\label{eq:Zn}
\begin{split}
\ud \|Z_t\|^{2(n+1)}&=-2(n+1)\|Z_t\|^{2n}\<Z_t,UZ_t\>\ud t+(\nicefrac{1}{2})\mathrm{Tr}[V^*H(Z_t)V]\ud t\\
&\qquad+2(n+1)\|Z_t\|^{2n}\<Z_t,V \ud B_t\>,
\end{split}
\end{equation}
where the matrix valued function $\mathbb{R}^d\ni x\mapsto H(x):=(H_{i,j}(x))_{i,j\in \{1,\ldots,d\}}\in \mathbb{R}^{d\times d}$ is given by
\begin{equation*}
\begin{split}
H_{i,j}(x):=
\begin{cases}
4(n+1)n \|x\|^{2(n-1)}x^2_i+2(n+1)\|x\|^{2n} &\quad  \textrm{ for } \quad i=j,\\
4(n+1)n \|x\|^{2(n-1)} x_ix_j &\quad \textrm{ for }\quad i\neq j.
\end{cases}
\end{split}
\end{equation*}
By definition of $\|\cdot \|_{1}$  it follows that
\begin{align}\label{eq:H11}
\|H(x)\|_{1}&\leq 2d(n+1)\|x\|^{2n}+
4d(n+1)n \|x\|^{2n}=2d(n+1)(1+2n)\|x\|^{2n}
\end{align}
for all $x\in \mathbb{R}^d$.
Note that $\mathrm{Tr}[V^*H(Z_t)V]=\mathrm{Tr}[H(Z_t)VV^*]$. 
By~\eqref{eq:H11} we obtain
\begin{equation}
\label{eq:trace101}
\begin{split}
|\mathrm{Tr}[H(Z_t)VV^*]|
&\leq 
d \|H(Z_t)\|_{1}\|VV^*\|_{\mathrm{F}}\leq 2d^2(n+1)(1+2n)\|VV^*\|_{\mathrm{F}} \|Z_t\|^{2n}.
\end{split}
\end{equation}
Using a localization argument in~\eqref{eq:Zn}
with the help of~\eqref{eq:matrix} and~\eqref{eq:trace101} yields
\begin{align*}
\frac{\ud}{\ud t} \mathbb{E}[\|Z_t\|^{2(n+1)}]\leq -2(n+1)\delta \mathbb{E}[\|Z_t\|^{2(n+1)}]+
d^2(n+1)(1+2n)\|VV^*\|_{\mathrm{F}}\mathbb{E}[\|Z_t\|^{2n}].
\end{align*}
By induction hypothesis we have 
$
\mathbb{E}[\|Z_t\|^{2n}]\leq C^n_* n!$  for all  $t\geq 0$.
Since $Z_0=0_d$, Lemma~\ref{lem:gron} yields for all $t\geq 0$
\begin{equation*}
\mathbb{E}[\|Z_t\|^{2(n+1)}]\leq \frac{d^2(n+1)(1+2n)\|VV^*\|_{\mathrm{F}}C^n_*\, n!}{2(n+1)\delta} 
\leq C^{n+1}_*(n+1)!,
\end{equation*}
which finishes the induction step.
This concludes the proof of item i).

We continue with the proof of item ii).
By the Monotone Convergence Theorem we have 
\begin{equation*}
\mathbb{E}[e^{\lambda \|Z_t\|^2}]=\sum\limits_{j=0}^{\infty} \frac{\lambda^j \mathbb{E}[\|Z_t\|^{2j}]}{j!}\quad \textrm{ for all } \quad 
\lambda\geq 0.
\end{equation*}
By item i) 
for all $\lambda \in (0,1/C_*)$ and $t\geq 0$ it follows that
\begin{equation*}
\mathbb{E}[e^{\lambda \|Z_t\|^2}]\leq \sum\limits_{j=0}^{\infty} (\lambda C_*)^j=\frac{1}{1-\lambda C_*}.
\end{equation*}
\end{proof}

\begin{lemma}[Covariance]\label{lem:covariance}
Assume that~\eqref{eq:matrix} holds and that the diffusion matrix $V$ is invertible.
Let $(Z_t)_{t\geq 0}$ be the unique strong solution of the 
SDE~\eqref{eq:Z990}
 and let $\Theta_t:=\mathbb{E}[ Z_t Z^*_t]$ for any $t\geq 0$. Then 
\[
\| \Theta_t-\Theta \|_{\mathrm{F}}\leq \|\Theta \|^2_{\mathrm{F}}\, e^{-2\delta t} \quad  \textrm{ for all } \quad t\geq 0,
\]
where  $\Theta\in \mathbb{R}^{d\times d}$
is the unique symmetric and positive definite solution of
the Lyapunov matrix equation
\begin{equation}\label{eq:LEM}
U\Theta+\Theta U^*=VV^*.
\end{equation}
\end{lemma}
\begin{proof}
The proof follows by analogous reasoning used in the proof of Lemma~C.4 in Appendix~C of~\cite{BJ1}. We state it here for completeness of the presentation.

Hypothesis~\eqref{eq:matrix}  and
Theorem~1, p. 443 of~\cite{LANTI}
yield that \eqref{eq:LEM} possesses a unique solution.
By Proposition~3.5 in~\cite{pavliotisbook} we have
\begin{equation}\label{eq:Theta}
\left\{
\begin{array}{r@{\;=\;}l}
\frac{\ud}{ \ud t}\Theta_t & -U\Theta_t-\Theta_t U^*+VV^*\quad \textrm{ for any } \quad t\geq 0,\\
\Theta_0 & 0_{d\times d},
\end{array}
\right.
\end{equation}
where $0_{d\times d}\in \mathbb{R}^{d\times d}$. Let $t\geq 0$ be fixed. Write $r_t:=\|\Theta_t-\Theta\|^2_{\mathrm{F}}$, 
$\Theta_t=(\Theta^{i,j}_t)_{i,j\in \{1,\ldots,d\}}$,
$\Theta=(\Theta^{i,j})_{i,j\in \{1,\ldots,d\}}$, $U=(U^{i,j})_{i,j\in \{1,\ldots,d\}}$ and 
$VV^*=((VV^*)^{i,j})_{i,j\in \{1,\ldots,d\}}$.
By~\eqref{eq:LEM} we obtain
\begin{equation}\label{eq:coord}
\sum_{k=1}^{d}
(U^{i,k}\Theta^{k,j}+\Theta^{i,k}U^{j,k})=(VV^*)^{i,j}\quad \textrm{ for all }\quad i,j\in \{1,\ldots,d\}.
\end{equation}
The differential equation~\eqref{eq:Theta} with the help 
of~\eqref{eq:coord}
reads 
\begin{equation}\label{eq:AB201}
\begin{split}
\frac{\ud }{\ud t}\Theta^{i,j}_t=\sum_{k=1}^{d}
(-
U^{i,k}\Theta^{k,j}_t-\Theta^{i,k}_tU^{j,k}+(VV^*)^{i,j}
)=-\sum_{k=1}^{d}
(
U^{i,k}(\Theta^{k,j}_t-\Theta^{k,j})+(\Theta^{i,k}_t-\Theta^{i,k})U^{j,k}
).
\end{split}
\end{equation}
The chain rule and~\eqref{eq:AB201} imply 
\begin{align*}
\frac{\ud}{\ud t} r_t&=2\sum_{i,j=1}^{d} (\Theta^{i,j}_t-\Theta^{i,j})\frac{\ud }{\ud t}(\Theta^{i,j}_t-\Theta^{i,j})\\
&=-2\sum_{i,j=1}^{d} (\Theta^{i,j}_t-\Theta^{i,j})\sum_{k=1}^{d}
(
U^{i,k}(\Theta^{k,j}_t-\Theta^{k,j})+(\Theta^{i,k}_t-\Theta^{i,k})U^{j,k}
)\\
&=-2\sum_{j=1}^{d}\sum_{i,k=1}^{d} (\Theta^{i,j}_t-\Theta^{i,j})
U^{i,k}(\Theta^{k,j}_t-\Theta^{k,j})-
2\sum_{i=1}^{d}\sum_{j,k=1}^{d}
(\Theta^{i,j}_t-\Theta^{i,j})
U^{j,k}(\Theta^{i,k}_t-\Theta^{i,k}),
\end{align*}
where in the last equality we rearrange the sums.
By~\eqref{eq:matrix} we deduce the following differential inequality
\begin{equation*}
\frac{\ud}{\ud t} r_t\leq 
-4\delta \sum_{i,j=1}^{d}(\Theta^{i,j}_t-\Theta^{i,j})^2=-4\delta  r_t\quad \textrm{ for all }\quad t\geq 0.
\end{equation*}
Lemma~\ref{lem:gron} yields
$r_t\leq r_0e^{-4\delta t}$ for all $t\geq 0$ and consequently  the statement.
\end{proof}

\section*{Acknowledgements}
The research of G. Barrera  has been supported by the Academy of Finland via
the Matter and Materials Profi4 university profiling action. He gratefully acknowledges support from a post-doctorate grant (2020-2023)  held
at the Department of Mathematical and Statistical Sciences at University of Helsinki
 and expresses his gratitude for all the facilities used along with the realization of this work.
The author is grateful to the reviewers
for the thorough examination of the manuscript, which has lead to a significant improvement.

\end{document}